\documentclass[11pt,article]{article}

\usepackage{amsmath}
\usepackage{amsthm}
\usepackage{paralist}
\usepackage{epsfig} 
\usepackage{graphicx}
\usepackage{caption}
\usepackage{subcaption}
\usepackage{graphics} 
\usepackage{epsfig} 
\usepackage{multirow}
\usepackage{array}
\usepackage{hhline}
\usepackage{epstopdf}
 \usepackage{multirow}
\input{amssym.tex}

\newtheorem{theorem}{Theorem}[section]

\newtheorem{lemma}[theorem]{Lemma}

 \numberwithin{equation}{section}
\theoremstyle{definition}
\newtheorem{definition}[theorem]{Definition}
\newtheorem{remark}{Remark}

\newcommand{\keywords}

\def\bc{\begin{center}}       \def\ec{\end{center}}
\def\ba{\begin{array}}        \def\ea{\end{array}}
\def\be{\begin{equation}}     \def\ee{\end{equation}}
\def\bea{\begin{eqnarray}}    \def\eea{\end{eqnarray}}
\def\beaa{\begin{eqnarray*}}  \def\eeaa{\end{eqnarray*}}

\def\dfrac#1#2{\frac{\displaystyle {#1}}{\displaystyle {#2}}}
\def\mathbb{\Bbb}

\begin{document}

\title{\bf Global solutions of a Keller--Segel system with saturated logarithmic sensitivity function}
\author{Qi Wang \thanks{(Email:{\tt qwang@swufe.edu.cn}).  This research is partially supported by the Fundamental Research Funds for the Central Universities, China.}\\
Department of Mathematics\\
Southwestern University of Finance and Economics\\
555 Liutai Ave, Wenjiang, Chengdu, Sichuan 611130, China
}

\date{}
\maketitle

\abstract
We study a Keller-Segel type chemotaxis model with a modified sensitivity function in a bounded domain $\Omega\subset \mathbb{R}^N$, $N\geq2$.  The global existence of classical solutions to the fully parabolic system is established provided that the ratio of the chemotactic coefficient to the motility of cells is not too large.

\textbf{Chemotaxis, global existence, logarithmic sensitivity.}

\textbf{AMS Subject Classification: Primary: 35B40, 35J57; Secondary: 92D25, 35B32, 35B35.}

\section{Introduction}

Chemotaxis is the oriented movement of cells along the gradient of certain chemicals in their environment.  One of the most interesting phenomena in chemotaxis is the aggregation of chemotactic cells.  Keller and Segel initiated the mathematical modeling of chemotaxis in their pioneering works \cite{KS,KS1,KS2} during 1970s.  Let $\Omega \subset \mathbb{R}^N$, $N \geq 1$ be a bounded domain and we denote $u(x,t)$ and $v(x,t)$ the cell population density and the chemical concentration respectively at space-time $(x,t)$.  Then a classical Keller-Segel chemotaxis model reads as follows
\begin{equation}\label{1}
\left\{
\begin{array}{ll}
 u_t=\nabla\cdot (d_1(u,v) \nabla u - \chi (u,v) \nabla \phi(v)) & x \in\Omega,~t>0,\\
v_t=d_2 (u,v) \Delta v+k(u,v)& x \in\Omega,~t>0,\\
u(x,0)=u_0(x)\geq 0,v(x,0)=v_0(x)\geq 0,& x\in\Omega,
\end{array}
\right.
\end{equation}
where $d_1>0$ is called the \emph{motility} of the cells and it interprets the ability of cells to move randomly.  $d_2>0$ is the diffusion rate of the chemical.  The chemotactic term $\chi(u,v)>0$ measures the strength of chemotactic response of cells to the chemical.  $\phi(v)>0$ is called the \emph{sensitivity function} and it reflects the variation of cellular sensitivity with respect to the levels of chemical concentration. $k(u,v)$ is the creation and degradation rate of the chemicals.

In this paper, we are concerned with the existence of global solutions of system (\ref{1}).  For the simplicity of our analysis and discussions, we assume that $d_1$ and $d_2$ are positive constants, and we choose $\chi(u,v)=\chi_0u$ where $\chi_0$ is a positive constant called \emph{chemotactic coefficient}.  Moreover, we assume that the kinetic term takes the form $k(u,v)=-c_1 v +c_2 u$ for some positive constants $c_1$ and $c_2$.  The choice of such cellular kinetics is selected in order to model the phenomenon that the cells secret the chemical which is consumed by certain enzyme.  Then the Keller-Segel system (\ref{1}) under the homogeneous Neumann boundary conditions takes the following form
\begin{equation}\label{2}
\left\{
\begin{array}{ll}
u_t=\nabla\cdot (d_1 \nabla u - \chi_0 u \nabla \phi(v)) & x \in\Omega, ~t>0,\\
v_t= d_2 \Delta v - c_1 v+c_2 u& x \in\Omega,~ t>0,\\
\dfrac{\partial u}{\partial \textbf{n}}=\dfrac{\partial v}{\partial \textbf{n}}=0, & x \in \partial \Omega,~t>0,\\
u(x,0)=u_0(x)\geq 0,~v(x,0)=v_0(x)\geq 0,& x\in \Omega,
\end{array}
\right.
\end{equation}
where $\partial \Omega$ is the smooth boundary of $\Omega$, $\textbf{n}$ is the unit outer normal to $\partial \Omega$.  The non-flux boundary conditions interpret an enclosed domain that inhibits both cell immigration and chemical flux across $\partial \Omega$.  Throughout the rest of this paper, we assume that the initial data $u_0$ and $v_0$ are not identical zeros.

In many cases, solutions of (\ref{2}) have been proved to exist globally in time--see the survey paper of Hillen and Painter \cite{HP}.   In particular, it is also well-known that the solutions of (\ref{2}), including a collection of its variations, in a 1D domain are always global and bounded in time according to the results in \cite{HPS, OY}.  However, such optimal result is not available for higher-dimensional domain and the sensitivity function $\phi(v)$ is an essential part in the existence of global solutions.  Two commonly utilized sensitivity functions are $\phi=v$ and $\phi=\ln v$, which lead (\ref{2}) to the so-called Minimal model and the logarithmic model, respectively.

For $\phi(v)=v$, $\Omega \subset \mathbb{R}^2$ and $d_1=d_2=1$, Nagai, etc. \cite{NSY} showed the existence of global solutions $(u(x,t),v(x,t))$ of (\ref{2}) provided that either $\int_\Omega u_0(x)dx < 4\pi/(c_2 \chi_0)$ or $\int_\Omega u_0(x)dx < 8\pi/(c_2 \chi_0)$ with $\Omega$ being a disk and $(u_0(x),v_0(x))$radial in $x$.  For $\phi(v)=v$, $\Omega \subset \mathbb{R}^N,N\geq3$, and $d_1=d_2=\chi_0=c_1=c_2=1$, Winkler \cite{Wk} established the global solutions of (\ref{2}), provided that $\Vert u_0(x) \Vert_{L^{N/2+\delta}}$ and $\Vert \nabla v_0(x) \Vert_{L^{N+\delta}} $ are small for any $\delta>0$.

Logarithmic model (\ref{2}) with $\phi(v)=\ln v$ has also attracted the attention of many authors.  For $\Omega \subset \mathbb{R}^2$ and $d_1=d_2=c_1=c_2=1$, Nagai, etc. \cite{NSY2} proved the existence of global solutions of (\ref{2}) if either $\chi_0<1$ or $\chi<5/2$ with $\Omega$ being a disk and $(u_0(x),v_0(x))$ radial in $x$.  Winkler \cite{Wk1} investigated the same problem in $\mathbb{R}^N,N\geq 2$ and established the global classical solutions if $\chi<\sqrt{2/N}$ and global weak solutions if $\chi<\sqrt{(N+2)/(3N-4)}$.  Moreover, Stinner and Winkler \cite{SW} showed the existence of global weak solutions of (\ref{2}) regardless of the size of $\chi$.  Their approach is based on the combination of some weighted integral estimates with a refined Hardy--Sobolev inequality.  On the other hand, a parabolic-elliptic system of (\ref{2}) was investigated by Nagai and Senba \cite{NS} and they established globally bounded solutions if either $N=2$, $\Omega$ is a disk, and $(u_0(x),v_0(x))$ is radial, or if $N\geq3$, $\Omega$ is a ball, and $(u_0(x),v_0(x))$ is radial while $\chi<2/(N-2)$.

The analysis of these results is heavily involved with the estimates of certain weighted functions of $u$ and $v$, which has been recently developed by D. Horstmann, M. Winkler et al.  See \cite{HW,Wk1,Wk2} and also \cite{Y} for details.  For the results concerning the blow-up solutions of (\ref{2}) and the global solutions of Keller-Segel models of other types, we refer readers to the surveys \cite{Ho,Ho2,HP}.

Though various forms of sensitivity functions can be chosen to model different types of chemotaxis, $\phi(v)=\ln v$ was selected largely due to the Weber-Fechner's law for cellular behaviors which states that, \emph{the subjective sensation is proportional to the logarithm of the stimulus intensity} \cite{HP}.  By an inspection of the logarithmic model $\phi(v)=\ln v$, one observes that the dynamics of the cellular movements are dominated by the taxis flux $\frac{\chi}{v} \nabla v$, which may become unbounded if $v \approx 0$.  Though it is not reasonable to assume that a low chemical concentration elicits a significant chemotactic response from the cell's motility, this singularity is a necessary mechanism for (\ref{2}) to generate travelling wave solutions and also has important applications.  See the last section of \cite{LW} for detailed discussions.  We would like to mention that, apart from the global solutions of the model, travelling waves are also interesting especially from the viewpoint of applications.  See \cite{LW,LW2} or the review paper \cite{W} by Z.A. Wang for works in this direction.

In this paper, we choose $\phi(v)=\ln(v+c)$ for a positive constant $c$ which has a damping effect on $\phi'(v)$ at $v=0$, and consider the following system over the smoothly bounded domain $\Omega \in \mathbb{R}^N$, $N\geq2$,
\begin{equation}\label{3}
\left\{
\begin{array}{ll}
u_t=\nabla\cdot (d_1 \nabla u - \dfrac{\chi_0 u}{v+c} \nabla v) & x \in\Omega, t>0,\\
v_t= d_2 \Delta v - c_1 v+c_2 u& x \in\Omega, t>0,\\
\dfrac{\partial u}{\partial \textbf{n}}=\dfrac{\partial v}{\partial \textbf{n}}=0, & x \in \partial \Omega,t>0,\\
u(x,0)=u_0(x)\geq 0,~v(x,0)=v_0(x)\geq 0,& x\in \Omega.
\end{array}
\right.
\end{equation}
For $N=2$ and $d_1=d_2=c_1=c_2=1$, Biler \cite{B} established the global solutions to a parabolic-elliptic system of (\ref{3}), where the $v$ equation is replaced by its stationary counterpart.   We aim to investigate the global existence of full parabolic--parabolic system (\ref{3}) for $N\geq2$ with arbitrary positive coefficients $d_1,d_2,c_1$ and $c_2$.  First of all, we introduce the scalings
\[d_1 t= \tilde t, \frac{\chi_0}{d_1}=\chi, \frac{d_2}{d_1}=k, \frac{c_1}{d_1}=\alpha, \frac{c_2}{d_1}=\beta,\]
and transform (\ref{3}) into
\begin{equation}\label{4}
\left\{
\begin{array}{ll}
u_t=\nabla\cdot ( \nabla u - \dfrac{\chi u}{v+c} \nabla v) & x \in\Omega, t>0,\\
v_t=k \Delta v - \alpha v+\beta u& x \in\Omega, t>0,\\
\dfrac{\partial u}{\partial \textbf{n}}=\dfrac{\partial v}{\partial \textbf{n}}=0, & x \in \partial \Omega,t>0,\\
u(x,0)=u_0(x)\geq 0,~v(x,0)=v_0(x)\geq 0,& x\in \Omega,
\end{array}
\right.
\end{equation}
where the tildes are dropped without confusing the reader.  Throughout the rest of this paper, we shall deal with (\ref{4}).  For the simplicity of notations, we denote
\[k_1=\left(1+4/N\right)-2\sqrt{2/N+4/N^2},~~k_2=(1+4/N)+2\sqrt{2/N+4/N^2}\]
 and \[\chi_1=(k-1)/2-\sqrt{2k/N},~~\chi_2=(k-1)/2+\sqrt{2k/N},\]
then the main theorem of our paper goes as follows.
\begin{theorem}\label{T1}  Let $\Omega$ be a bounded domain in $\mathbb{R}^N$, $N\geq2$ and $(u_0(x),v_0(x)) \in W^{1,p}(\Omega)\times W^{2,p}(\Omega)$ for some $p>N$.  Assume that $\chi\in(0,\chi_2)$ if $k\in(k_1,k_2)$ and $\chi \in (\chi_1,\chi_2)$ if $k\in [k_2,\infty)$.  Then there exists a unique classical global solution $(u(x,t),v(x,t))$ to (\ref{4}) for arbitrary positive constants $\alpha, \beta$ and $c$.
\end{theorem}

\begin{remark} In \cite{NS, NSY, NSY2, Wk, Wk1, Wk2}, it seems that the condition $d_1=d_2=1$ in (\ref{2}) (at least $d_1=d_2$, i.e., $k=1$ in (\ref{4})) can not be dropped due to technical reasons.  Theorem \ref{T1} does not require $k=1$, since we always have $k_1<1<k_2$ for all $N>1$.  However, if $k=1$, we readily see that (\ref{4}) admits global solutions if $\chi<\chi_2=\sqrt{2/N}$, which coincides with the upper bound for $\chi$ in \cite{Wk1}.  We want to point out that, a lower bound $\chi_1$ is need for $\chi$ in Theorem \ref{T1} for $k$ being large, however, this assumption is only technically needed by our approach and we believe this unusual lower bound can be relaxed.
\end{remark}
We would also like to mention that, for $\chi(v)\leq\frac{\chi}{(v+c)^p}$ with $p>1$, Winkler \cite{Wk2} investigated a similar system and obtained a uniform-in-time bounded solution for all $\chi>0$ in any dimensions.  However, the borderline $p=1$ has not been studied due to technical issues and it is the goal of this paper to study the system under this crucial condition.  Our work is motivated by that of Winkler \cite{Wk2} and we apply a technique developed by Winkler in \cite{Wk,Wk1,Wk2}, etc. though there are some difficulties to overcome since our results apply for all positive $\alpha,\beta$ and all $k>k_1$.  Our analysis also covers the one-dimensional problem, however, it has been shown that globally bounded solutions exist for $\chi\in(0,\infty)$ in this case.

Throughout this paper, we assume that $C$ is a generic positive constant that may vary from line to line, unless otherwise noted.

\section{Existence of global solutions}

\subsection{Preliminaries}

First of all, we collect some basic properties of an analytic semigroup subject to the homogeneous Neumann boundary condition.  Then we proceed to obtain local solutions to (\ref{4}) by a standard fixed point argument.  These results are classical and well-known and we refer readers to \cite{H} for the classical results and \cite{HW, Wk} for the recent developments.
We denote the linear operator $A_k=-k\Delta+1, k>0$ and $A=-\Delta+1$.  We shall only collect the properties of $A$ here while the same properties for $A_k$ will be applied in the coming analysis.  It is well known that $A$ is sectorial in $L^p(\Omega)$ and it possesses closed fractional powers $A^\theta, \theta \in (0,1)$.  The domain $\mathcal{D}(A^\theta)$ equipped with norm
\[\Vert w \Vert_{\mathcal{D}(A^\theta)} =\Vert A^\theta w \Vert_{L^p(\Omega)}\]
is a Banach space and for $m=0,1$, $p\in[1,\infty]$, $q\in(1,\infty)$, we have the following embedding
\begin{equation}\label{5}
\mathcal{D}(A^\theta) \hookrightarrow
\left\{
\begin{array}{ll}
W^{m,q}(\Omega),&\text{ if } m-N/q<2\theta-N/p, \\
C^\delta(\bar{\Omega}),& \text{ if } 2\theta-N/p >\delta\geq0.
\end{array}
\right.
\end{equation}
Moreover, $(e^{t\Delta})_{t\geq 0}$ maps $L^p,p\geq1$ into $\mathcal{D}(A^\theta)$ and for any $q\in(p,+\infty]$, we have the following estimates
\begin{equation}\label{6}
\Vert A^\theta e^{-At} w \Vert_{L^q(\Omega)} \leq C t^{-\theta-\frac{N}{2}(\frac{1}{p}-\frac{1}{q})} e^{-\nu t}\Vert w \Vert_{L^p(\Omega)}, \forall~w \in L^p(\Omega),
\end{equation}
and
\begin{equation}\label{7}
\Vert A^\theta e^{t\Delta} w \Vert_{L^q(\Omega)} \leq C t^{-\theta-\frac{N}{2}(\frac{1}{p}-\frac{1}{q})} e^{-\nu t}\Vert w \Vert_{L^p(\Omega)}, \forall~w \in L^p(\Omega),\int_\Omega w=0,
\end{equation}
where $C$ depends on $\Omega$, $N$ and $p$ and $\nu>0$ is the first nonzero eigenvalue of $-\Delta$ with Neumann boundary condition.  Furthermore, if $1\leq p\leq q \leq \infty$, there exists a positive constant $C>0$ such that, for all $t>0$
\begin{equation}\label{8}
\Vert \nabla e^{t\Delta} w \Vert_{L^q(\Omega)} \leq C\Big(1+ t^{-\frac{1}{2}-\frac{N}{2}(\frac{1}{p}-\frac{1}{q})}\Big) e^{-\nu t}\Vert w \Vert_{L^p(\Omega)}, \forall~w \in L^p(\Omega).
\end{equation}

\subsection{Local existence}

In this section, we prove the existence and uniqueness of local solutions of (\ref{4}) by the standard Banach fixed point theorem.  Our proof of the local existence starts with a notion of weak solutions and the classical solutions can be established through parabolic regularity theory.  To this end, we convert (\ref{4}) into the following integral forms
\begin{align}\label{9}
u(\cdot,t)=&\left. e^{-At}u_0\!-\!\int_0^t \nabla \cdot e^{-A(t-s)}\left(\frac{\chi u(\cdot,s)}{v(\cdot,s)\!+\!c} \nabla v(\cdot,s) \right) ds\!+\!\int_0^t  e^{-A(t-s)} u(\cdot,s) ds, \right. \nonumber\\
v(\cdot,t)=&\left.e^{-A_kt}v_0+\int_0^t e^{-A_k(t-s)} \Big((1-\alpha)v(\cdot,s)+\beta u(\cdot,s) \Big)ds, \right.
\end{align}
and we shall show that (\ref{9}) possesses a fixed point in a neighborhood of $(u_0,v_0)$.

\begin{lemma}\label{lem1}
Let the initial data $u_0(x), v_0(x)\geq 0$ be not identically zeros and $(u_0,v_0) \in C(\bar{\Omega}) \times (W^{1,p}(\Omega))$ for some $p>N$.  Then there exists $T_{\max}\in(0,\infty]$ such that
(\ref{4}) has a unique nonnegative solution $(u(x,t),v(x,t))$ which is continuous in $\bar \Omega \times (0,T_{\max})$ with $(u(x,t),v(x,t))\in C^{2,1}\big(\bar{\Omega}  \times (0,T_{max})\big) \times L^\infty_{loc}\big([0,T_{max});W^{1,p}(\Omega)\big)$ and $(u,v)\in C^{2,1}\big( \bar{\Omega} \times (0,T_{max})\big)$, where $T_{\max}$ is the so-called maximal existence time; furthermore, either $T_{\max}=\infty$ or $T_{\max}<\infty$ and
\[\lim_{t \rightarrow T^-_{\max}}   \Vert u(\cdot,t)\Vert_{L^\infty(\Omega)}+\Vert v(\cdot,t)\Vert_{L^\infty(\Omega)}=\infty.\]
\end{lemma}

\begin{proof}  For each fixed $T>0$, we define the operator
\[\mathcal{T}:
\begin{pmatrix}
u\\
v
\end{pmatrix}
\rightarrow
\begin{pmatrix}
\Phi\\
\Psi
\end{pmatrix},
\]
where $\Phi$ is the given as the right hand side of the first equation in (\ref{9}) and $\Psi$ is the right hand side in its second equation.  Furthermore, we introduce the function space
\[X:=\{(u,v) ~\vert~ (u,v)\in C\big([0,T];C(\bar{\Omega})\big) \times C\big([0,T];W^{1,p}(\Omega)\big)  \} ,\]
equipped with norm $\Vert (u,v) \Vert _X= \max_{s\in[0,t]}\Vert u \Vert_{C(\bar{\Omega})}+\max_{s\in[0,t]}\Vert v \Vert_{W^{1,p}(\Omega)}, t\in [0,T]$.  For each fixed $(u_0,v_0)$, we want to show that the unit ball $B(u_0,v_0)$ in $X$ is mapped into itself by $\mathcal{T}$ for all $t\in (0,T)$.

Taking any two pairs $(u_1,v_1)$ and $(u_2,v_2)$ in $B(u_0,v_0)$, we have that
\[\Big\vert \frac{u_1 \nabla v_1}{v_1+c} -\frac{u_2 \nabla v_2}{v_2+c}\Big\vert=\Big\vert\frac{u_1 \nabla v_1}{v_1+c} -\frac{u_1 \nabla v_2}{v_1+c}+\frac{u_1\nabla v_2}{v_1+c} -\frac{u_2 \nabla v_2}{v_2+c}\Big\vert\leq C \Big\vert(u_1-u_2)+\nabla(v_1-v_2)\Big\vert,\]
thus we conclude from the estimates (\ref{5})--(\ref{8}) that, for each $p>N$, there exists some $ {\theta_1} \in (\frac{N}{2p},\frac{1}{2})$ such that
\begin{align*}
  \begin{split}
  & \Vert \Phi(u_2,v_2)-\Phi(u_1,v_1) \Vert_{C(\bar{\Omega})} \\
 \leq &  \Big\Vert \int_0^t \nabla \cdot e^{-A(t-s)} \left(\frac{\chi u_1}{v_1+c} \nabla v_1-\frac{\chi u_2}{v_2+c} \nabla v_2 \right) ds \Big\Vert_{C(\bar{\Omega})}\\
&+\Big\Vert  \int_0^t \nabla \cdot e^{-A(t-s)} (u_2-u_1) ds \Big\Vert_{C(\bar{\Omega})}\\
  \leq &  C \Big(\int_0^t  \Vert A^{\theta_1} e^{-A(t-s)}  \left(\frac{\chi u_1}{v_1+c} \nabla v_1-\frac{\chi u_2}{v_2+c} \nabla v_2 \right) \Vert_{C(\bar{\Omega})} ds\\
&+\int_0^t  \Vert A^{\theta_1} e^{-A(t-s)}  (u_2-u_1) \Vert_{C(\bar{\Omega})} ds\Big)
   \end{split}
\end{align*}
\begin{align}\label{10}
  \begin{split}
 \leq &  C \Big( \int_0^t (t-s)^{-\frac{1}{2}-\theta_1} e^{-\nu (t-s)} \Vert v_2-v_1 \Vert_{W^{1,p}(\Omega)} ds \\
 &+ \int_0^t (t-s)^{-\frac{1}{2}-\theta_1} e^{-\nu (t-s)} \Vert u_2-u_1 \Vert_{C(\bar{\Omega})} ds \Big) \\
  \leq& C t^{\frac{1}{2}-\theta_1} \Big(\Vert v_2-v_1 \Vert_{W^{1,p}(\Omega)}+\Vert u_2-u_1 \Vert_{C(\bar{\Omega})}\Big);
  \end{split}
\end{align}
moreover, there exists $\theta_2 \in (\frac{N}{2p},1)$ and $\theta_3\in(\frac{1}{2},1)$ such that
\begin{align}\label{11}
  \begin{split}
&\Vert \Psi(u_2,v_2)-\Psi(u_1,v_1) \Vert_{W^{1,p}(\Omega)}\\
=& \Big\Vert\int_0^t e^{-A_k(t-s)} \Big((1-\alpha)(v_2-v_1)+\beta(u_2-u_1) \Big)  ds \Big\Vert_{W^{1,p}(\Omega)}\\
\leq &C  \left( t^{1-{\theta_2}} \Vert u_2-u_1 \Vert_{C(\bar{\Omega})}+t^{1-{\theta_3}} \Vert v_2-v_1 \Vert_{W^{1,p}(\Omega)}  \right)
  \end{split}
\end{align}
Then it follows from (\ref{10}) and (\ref{11}) that, for each $p>N$, there exists some $\theta_4\in(0,1)$ such that
\[
\begin{split}
&\Vert (\Phi(u_2,v_2),\Psi(u_2,v_2))-(\Phi(u_1,v_1),\Psi(u_1,v_1)) \Vert_X \\
\leq &C t^{1-{\theta_4}}( \Vert u_2-u_1\Vert_{C(\bar{\Omega})}+  \Vert v_2-v_1 \Vert_{W^{1,p}(\Omega)}).
\end{split}\]
By taking $T$ sufficiently small, we can easily see that
\[\Vert (\Phi(u_2,v_2),\Psi(u_2,v_2))-(\Phi(u_1,v_1),\Psi(u_1,v_1)) \Vert_X \leq \frac{1}{2} \Vert (u_2,v_2)-(u_1,v_1) \Vert_X,\]
for all $t\in [0,T)$, hence we conclude from the Banach fixed point theorem that the operator $\mathcal{T}$ has a unique fixed point in $B(u_0,v_0)$ ,which is apparently a weak solution of (\ref{4}).  By the standard parabolic regularity arguments in \cite{LSU}, we can show that $(u(x,t),v(x,t))$ is a classical solution and it satisfies the regularity properties in the lemma.  Moreover, both $u(x,t)$ and $v(x,t)$ are nonnegative in $\Omega\times (0,T_{max})$ thanks to the maximum principles.

To show the uniqueness of the solution to (\ref{4}), we proceed as in \cite{Wk2} and let $w=u_1-u_2$ and $z=v_1-v_2$.  Then it follows from straightforward calculations that
\[\frac{1}{2}\frac{d}{dt}\int_\Omega w^2 dx+\int_\Omega \vert \nabla w \vert^2 dx=\chi\int_\Omega \Big(\frac{u_1\nabla v_1}{v_1+c}-\frac{u_2\nabla v_2}{v_2+c}\Big)\nabla w dx  \]
and
\[\frac{1}{2}\frac{d}{dt}\int_\Omega  \vert \nabla z \vert^2dx+k\int_\Omega  \vert \Delta z \vert^2dx+\alpha \int_\Omega  \vert \nabla z \vert^2dx=-\beta\int_\Omega w \Delta z dx.\]
Then we have from Young's inequality and the estimates on $u_i,\nabla v_i $, $i=1,2$ that there exist $c(T)>0$
\[\frac{d}{dt}\Big(\int_\Omega  \vert \nabla z \vert^2dx+\int_\Omega z^2 dx+ \int_\Omega w^2 dx\Big)\leq c(T) \Big(\int_\Omega  \vert \nabla z \vert^2dx+\int_\Omega z^2 dx+ \int_\Omega w^2 dx\Big)  \]
for all $t\in(0,T)$.  Then we have from the Gronwall's lemma that $w\equiv0$ and $z\equiv 0$ in $\Omega \times [0,T]$.  This shows the uniqueness and concludes the proof of this lemma.
\end{proof}

\subsection{Global existence}

In this section, we present the proof of Theorem \ref{T1}.  To this end, it is equivalent to show that the local solution obtained in Lemma \ref{lem1} exists for all $t\in(0,\infty)$ if $k>k_1$ and $\chi\in I^+(N/2)$, which are defined as below.   We want to remind our readers that the quadratic function $f(p)=(\chi-\frac{k-1}{2})^2-\frac{k}{p}$ plays an essential role in our analysis.
\begin{definition}\label{def22}
We define
\[I^+(p)=\Big\{\chi>0 ~\Big\vert~ f(p)=\left(\chi-\frac{k-1}{2}\right)^2-\frac{k}{p}<0 \Big\};\]
in particular, for $p=N/2$ we have
\begin{equation*} I^+(N/2)=
\left\{
\begin{array}{ll}
(0,\chi_2), & \text{ if } k\in(k_1,k_2),\\
(\chi_1,\chi_2), & \text{ if } k\in[k_2,\infty),
\end{array}
\right.
\end{equation*}
where $k_1=\left(1+4/N\right)-2\sqrt{2/N+4/N^2},~~k_2=(1+4/N)+2\sqrt{2/N+4/N^2}$, and $\chi_1=(k-1)/2-\sqrt{2k/N},~~\chi_2=(k-1)/2+\sqrt{2k/N}$.
\end{definition}

\begin{remark}Since $f(p)$ is monotone increasing in $p$, we readily see that $I^+(p_1) \subset I^+(p_2)$ if $p_1>p_2$.  $k_1, k_2$ are both positive and $\chi_1$, $\chi_2$ are the two roots of $f(N/2)=0$.  In general the set of $\chi$ for $f(N/2)<0$ takes the form $(\chi_1,\chi_2)$, however, we see that $\chi_1<0$ if $k\in(k_1,k_2)$ and $\chi_1<\chi_2<0$ if $k\in(0,k_1)$.  Our analysis does not work for model (\ref{4}) when $k\in(0,k_1)$ with $\chi<0$ between $\chi_1$ and $\chi_2$, where (\ref{4}) becomes a chemo-repulsion model.  On the other hand, we see that $\chi_1,\chi_2 \rightarrow 0^-$ as $k\rightarrow 0^+$, so our analysis only covers a very tiny range for $\chi<0$ if $k$ is small.  In this paper, we only study the model (\ref{4}) with $\chi>0$, which represents the chemoattraction of the chemical, and we shall assume that $k>k_1$ from now on.
\end{remark}

Theorem \ref{T1} is a consequence of several lemmas.  We shall first present the following result on the global boundedness of the chemical concentration $v(x,t)$ and the $L^p$-estimate on $u(x,t)$, the cell population density for some pre-determined $p$.
\begin{lemma}\label{lem2}
Let $N\geq2$ and we assume that $k\in(k_1,\infty)$ and $\chi \in I^+(N/2)$ as in Theorem \ref{T1}, then there exists a $C(t)$ such that the solution $v(x,t)$ of (\ref{4}) satisfies
\begin{equation}\label{12}
\sup_{s\in(0,t)}\Vert v(\cdot,s) \Vert _{L^\infty}  \leq C(t), \forall t \in(0,T_{\max});
\end{equation}
moreover, for all $p \in \left(\frac{N}{2}, \Big(\frac{k}{(\chi-\frac{k-1}{2})^2}\Big)^+\right)$, where we choose the conventional notation
\begin{equation*}\left(\frac{k}{(\chi-\frac{k-1}{2})^2}\right)^+=
\left\{
\begin{array}{ll}
\Big(\frac{k}{(\chi-\frac{k-1}{2})^2}\Big)& \text{ if } \chi\neq  \frac{k-1}{2} ,\\
+\infty,& \text{ if } \chi= \frac{k-1}{2} ,
\end{array}
\right.
\end{equation*}
there exists $C(t)$ such that the solution $u(x,t)$ of (\ref{4}) satisfies
\begin{equation}\label{13}
\sup_{s\in(0,t)}\Vert u(\cdot,s) \Vert _{L^p}  \leq C(t), \forall t \in(0,T_{\max}),
\end{equation}
where $T_{\max}\in(0,\infty)$ is the maximal existence time obtained in Lemma \ref{lem1}.
\end{lemma}

\begin{proof}  For notational convenience, we denote
\[\chi(v)=\frac{\chi}{v+c}.\]
Then for any $p \in \Big(\frac{N}{2}, \Big(\frac{k}{(\chi-\frac{k-1}{2})^2}\Big)^+\Big)$, we test $u^p$ by $\varphi(v)=(v+c)^{\frac{1-p}{2}}$ and we obtain through integration by parts that
\begin{align}\label{14}
  \begin{split}
 &\frac{1}{p} \frac{d}{dt} \int_\Omega u^p \varphi(v) dx=\int_\Omega u^{p-1} \varphi(v) u_t dx +\frac{1}{p} \int_\Omega u^p \varphi'(v) v_t dx \\
=&  \int_\Omega u^{p-1} \varphi(v) \Delta u dx-\int_\Omega u^{p-1} \varphi(v) \nabla \cdot (u\chi(v) \nabla v) dx+\frac{k}{p} \int_\Omega u^p \varphi'(v) \Delta v dx\\
&+\frac{1}{p}\int_\Omega (\beta u^{p+1}-\alpha u^p v) \varphi'(v)dx \\
= &-\!(p\!-\!1)\!\!\int_\Omega\!\!\! u^{p-2} \varphi(v) \vert \nabla u \vert^2 dx\!+\! \int_\Omega\!\! \Big((p\!-\!1)\varphi(v) \chi(v)\!-\!(k\!+\!1)\varphi'(v)\Big)u^{p-1} \nabla u \nabla v dx\\
&  +\int_\Omega \Big(\varphi'(v)\chi(v)-\frac{k}{p} \varphi''(v)\Big) u^p \vert \nabla v \vert^2 dx+ \frac{1}{p} \int_\Omega (\beta u^{p+1}-\alpha u^p v) \varphi'(v) dx.
  \end{split}
\end{align}
By applying Young's inequality in (\ref{14}), we have that
\begin{equation}\label{15}
\begin{split}
  & \left. \int_\Omega \Big((p-1)\varphi(v) \chi(v)-(k+1)\varphi'(v)\Big)u^{p-1} \nabla u \nabla v dx \right.\\
\leq    & \left.  \!(p\!-\!1)\!\!\!\int_\Omega \!\!u^{p-2} \varphi(v) \vert \nabla u \vert^2 dx\!+ \!\!\!\int_\Omega\!\! \frac{ \big((p\!-\!1)\varphi(v) \chi(v)\!-\!(k\!+\!1)\varphi'(v)\big)^2u^p \vert \nabla v \vert^2}{4(p-1)\varphi(v)} dx. \right.
\end{split}
\end{equation}
Moreover, for all $p\in(1,\infty)$, we claim that
\begin{equation}\label{16}
\frac{ \Big((p-1)\varphi(v) \chi(v)-(k+1)\varphi'(v)\Big)^2}{4(p-1)\varphi(v)} +\left(\varphi'(v)\chi(v)-\frac{k}{p} \varphi''(v)\right) < 0,
\end{equation}
To prove this, we substitute
\[\varphi'(v)=\frac{1-p}{2}(v+c)^{-\frac{p+1}{2}},~\varphi''(v)=\frac{p^2-1}{4}(v+c)^{-\frac{p+3}{2}},\]
and $\chi(v)=\frac{\chi}{v+c}$ into (\ref{16}), then we readily see that it is equivalent to the following quadratic inequality
\begin{equation}\label{17}
f(p)=\left(\chi-\frac{k-1}{2} \right)^2-\frac{k}{p}<0,
\end{equation}
which obviously holds for $p \in \left(\frac{N}{2}, \Big(\frac{k}{(\chi-\frac{k-1}{2})^2}\Big)^+\right)$, therefore we have proved our claim.  After balancing out the $\nabla u \nabla v $ terms and making the coefficient of $\vert \nabla v \vert^2 $ negative from (\ref{16}), we observe that (\ref{14}) becomes
\begin{equation}\label{18a}
\begin{split}
 &\frac{d}{dt} \int_\Omega u^p (v+c)^{\frac{1-p}{2}} dx \\
 \leq &\frac{\alpha(p-1)}{2} \int_\Omega u^p  (v+c) ^{\frac{1-p}{2}}dx-\frac{\beta(p-1)}{2}\int_\Omega u^{p+1} (v+c)^{-\frac{p+1}{2}} dx.
 \end{split}
\end{equation}
Putting
\[y_p(t)=\int_\Omega u^p(\cdot,t)  (v(\cdot,t)+c)^\frac{1-p}{2} dx,\]
we see that (\ref{18a}) implies that
\[y'_p(t) \leq \frac{\alpha(p-1)}{2}y_p(t),~\text{while } y_p(0)=\int_\Omega u_0^p  (v_0+c)^\frac{1-p}{2} dx.\]
Solving this differential inequality gives us
\begin{equation}\label{18}
y_p(t) \leq y_p(0) e^{\frac{\alpha(p-1)}{2}t}, t\in(0,\infty),
\end{equation}
then we have from (\ref{18}) that $y_p(t)$ is bounded on $(0,t)$ for all $t<\infty$.  On the other hand, we recall $v(x,t)$ from its integral representation form
\[v(\cdot,t)=e^{-A_kt}v_0 + \int_0^t e^{-A_k(t-s)} \Big(\beta u(\cdot,s)+(1-\alpha)v(\cdot,s)  \Big) ds,~t\in(0,\infty).\]
Applying estimate (\ref{6}) with $q=\infty$ to the abstract form above, we obtain that
\begin{equation}\label{19}
 \Vert v(\cdot,t) \Vert_{L^\infty}\! \leq C\! \Big(\!\Vert v_0 \Vert_{L^\infty}\!+\!\int_0^t\!\! (t-s)^{-\frac{N}{2p}} e^{-\nu(t-s)} \big(\beta \Vert u(\cdot,s) \Vert_{L^p}+\vert \alpha-1\vert \Vert v(\cdot,s) \Vert_{L^p}  \big)  ds \Big).
\end{equation}
For the integrand on the right hand side of (\ref{19}), we have that
\begin{equation}\label{20}
\Vert u(\cdot,t) \Vert_{L^p} \leq y_{2p}^{\frac{1}{2p}}(t) \cdot \Vert v(\cdot,t)+c \Vert_{L^{p-\frac{1}{2}}} ^\frac{p-\frac{1}{2}}{2p}, t\in(0,\infty),
\end{equation}
To show (\ref{20}), we have that
\[\int_\Omega u^p=\int_\Omega u^p(v+c)^{\frac{1-2p}{4}}(v+c)^{\frac{2p-1}{4}} \leq \left( \int_\Omega u^{2p}(v+c)^{\frac{1-2p}{2}} \right)^\frac{1}{2}  \left(\int_\Omega (v+c)^{\frac{2p-1}{2}} \right)^\frac{1}{2}, \]
where the last inequality follows from Holder's inequality.  By taking the $p$-th roots on both hand sides of the inequality above, we conclude that
\[\left(\int_\Omega u^p \right)^\frac{1}{p} \leq \left( \int_\Omega u^{2p}(v+c)^{\frac{1-2p}{2}} \right)^\frac{1}{2p}  \left( \int_\Omega(v+c)^{\frac{2p-1}{2}} \right)^\frac{1}{2p}, \]
which is the exact inequality we desire in (\ref{20}).

We are now ready to show that $\Vert v(\cdot,t) \Vert_{L^\infty}$ is bounded for all $t\in(0,\infty)$.  To this end, we let
\[M(\tau)=\sup_{t\in(0,\tau) } \Vert v(\cdot,t) \Vert_{L^\infty}, \forall \tau >0,\]
then by (\ref{19}) we have that for all $\tau>0$,
\begin{equation}\label{21}
M(\tau)\!\leq \!C\!\left(\!\Vert v_0 \Vert_{L^\infty}\!+\!\frac{\vert \alpha\!-\!1 \vert }{1\!-\!\frac{N}{2p}}\tau^{1-\frac{N}{2p}} M(\tau)+\frac{\beta}{1\!-\!\frac{N}{2p}}\tau^{1-\frac{N}{2p}} \!\!\sup_{t\in(0,\tau)} \!\! \Vert u(\cdot,s) \Vert_{L^p}\!\! \right),
\end{equation}
where $C$ is independent of $\tau$.  Moreover, we have from (\ref{20}) that
\begin{equation}\label{22}
\sup_{t\in(0,\tau)}  \Vert u(\cdot,s) \Vert_{L^p}\leq \sup_{t\in(0,\tau) } \Vert v(\cdot,t)+c \Vert_{L^\infty}^\frac{p-\frac{1}{2}}{2p} y_{2p}^\frac{1}{2p}(\tau)\leq \left(M^\frac{p-\frac{1}{2}}{2p}+c^\frac{p-\frac{1}{2}}{2p} \right) y_{2p}^\frac{1}{2p}(\tau),
\end{equation}
therefore we can take $\tau$ small and conclude from (\ref{21}) and (\ref{22}) that
\[M(\tau) \leq C \Big(\Vert v_0 \Vert_{L^\infty},N,\Omega \Big),\]
i.e., $\Vert v(\cdot,t) \Vert_{L^\infty}$ is bounded in $[0,\tau]$.  Repeating this argument on time interval $[\tau,2\tau],[2\tau,3\tau],...$, we see that $\Vert v(\cdot,t) \Vert_{L^\infty}$ is bounded on any finite time interval.  Therefore we have established the global boundedness of the chemical concentration $v(x,t)$, i.e, $\Vert v(\cdot,t) \Vert_{L^\infty}\leq C(t), t\in(0,\infty).$  Moreover, we readily have from (\ref{20}) that
\begin{equation*}
\Vert u(\cdot,t) \Vert_{L^p} \leq y_{2p}^{\frac{1}{2p}}(t) C(t)\leq C(\Vert u_0 \Vert_{L^\infty},\Vert v_0 \Vert_{L^\infty})e^{\frac{\alpha(2p-1)}{4p}t}\leq C(t), t\in(0,\infty).
\end{equation*}
Thus we have proved Lemma \ref{lem2}.
\end{proof}

Lemma \ref{lem2} provides an estimate on $\Vert u \Vert_{L^p}$ and to establish the boundedness of $\Vert u \Vert_{L^\infty}$, we need an estimate on the norm of $\nabla v$, which is presented in the following lemma.
\begin{lemma}\label{lem3}
Let $N\geq2$ and we assume that $k\in(k_1,\infty)$ and $\chi \in I^+(N/2)$, then there exists $C(t)>0$ such that,

if $\chi \in I^+(N)$,
\begin{equation}\label{24}
\sup_{s\in(0,t)} \Vert \nabla v(\cdot,s) \Vert _{L^\infty}  \leq C(t)<\infty, t\in(0,T_{\max})
\end{equation}
and if $\chi \in I^+(N/2)\backslash I^+(N)$,
\begin{equation}\label{25}
\sup_{s\in(0,t)} \Vert \nabla v(\cdot,s) \Vert _{L^q}  \leq C(t)<\infty, t\in(0,T_{\max})
\end{equation}
for all $q< \Big(\frac{k}{(\chi-\frac{k-1}{2})^2-\frac{k}{N}}\Big)^+$, where $(\cdot)^+$ is the same conventional notation as in Definition \ref{def22}.
\end{lemma}

\begin{proof}  We choose $p=\Big((\chi-\frac{k-1}{2})^2/k+\frac{\epsilon}{N}\Big)^{-1}$, where $\epsilon>0$ is taken to be an arbitrarily small but fixed number, then it follows from Lemma \ref{lem2} that
\[\Vert u(\cdot,t) \Vert_{L^p(\Omega)} \leq C(t), ~\Vert v(\cdot,t) \Vert_{L^\infty(\Omega)} \leq C(t), \forall t \in (0,\infty).\]
Now, for all $q\in(p,\infty]$, we have from the smoothing estimates that, there exists a constant $C>0$ independent of $t$ such that
\begin{equation}\label{26}
\begin{split}
\Vert v(\cdot,t)& \Vert _{W^{1,q}(\Omega)} \leq C\Big(\Vert v_0 \Vert _{W^{1,q}(\Omega)} \\
+ &\int_0^t (t-s)^{-\frac{1}{2}-\frac{N}{2}(\frac{1}{p}-\frac{1}{q})} \Big(\beta \Vert u(\cdot,s) \Vert_{L^p}+\vert \alpha-1 \vert \Vert v(\cdot,s) \Vert_{L^p}\Big) ds\Big).
\end{split}
\end{equation}
Our arguments are divided into the following two cases:

if $\chi \in I^+(N)$, we take $q=\infty$ in (\ref{26}) and it leads us to
\begin{align}\label{27}
  \begin{split}
 \Vert \nabla v(\cdot,t) \Vert _{L^{\infty}(\Omega)} & \left.  \leq C \left(\Vert v_0 \Vert _{W^{1,\infty}(\Omega)}\!+\!\int_0^t (t\!-\!s)^{-\frac{1}{2}-\frac{N}{2p}} \Big(\Vert u(\cdot,s) \Vert_{L^p}\!+\!\Vert v(\cdot,s) \Vert_{L^p} \Big) ds\right) \right.\\
   & \left.  \leq C \left(\Vert v_0 \Vert _{W^{1,\infty}(\Omega)}+C(t) \int_0^t (t-s)^{-\frac{1}{2}-\frac{N}{2p}}  ds\right)  , \right.
  \end{split}
\end{align}
Since $\chi \in I^+(N)$, we have that $(\chi-\frac{k-1}{2})^2 < \frac{k}{N}$ and
\[p=\frac{k}{(\chi-\frac{k-1}{2})^2+\epsilon}>\frac{k}{\frac{k}{N}+\epsilon}>\frac{N}{2},\]
which implies that $-\frac{1}{2}-\frac{N}{2p}>-1$, then we must have from (\ref{27}) that
\begin{equation*}
\Vert \nabla v(\cdot,t) \Vert _{L^\infty}  \leq C(t), t\in(0,\infty);
\end{equation*}

if $\chi \in I^+(N/2)\backslash I^+(N)$.  Then similar as above, we have that
\begin{equation}\label{28}
\begin{split}
\Vert v(\cdot,t) \Vert _{W^{1,q}(\Omega)} \leq C\left(\Vert v_0 \Vert _{W^{1,q}(\Omega)} + C(t) \int_0^t (t-s)^{-\frac{1}{2}-\frac{N}{2}(\frac{1}{p}-\frac{1}{q})}   ds\right).
\end{split}
\end{equation}
To show (\ref{25}), it is equivalent to prove $-\frac{1}{2}-\frac{N}{2}(\frac{1}{p}-\frac{1}{q})>-1$ for all \\$q< \Big(\frac{k}{(\chi-\frac{k-1}{2})^2-\frac{k}{N}}\Big)^+$.  To this end, we can choose $\epsilon$ to be sufficiently small such that $q<  \frac{k}{(\chi-\frac{k-1}{2})^2-\frac{k}{N}+2 \frac{k\epsilon}{N} }$, then we have that
\[
\begin{split}
-\frac{1}{2}-\frac{N}{2}\left(\frac{1}{p}-\frac{1}{q}\right)&>-\frac{1}{2}-\frac{N}{2}\left(\frac{(\chi-\frac{k-1}{2})^2+\frac{k\epsilon}{N}}{k}-
\frac{(\chi-\frac{k-1}{2})^2-\frac{k}{N}+2\frac{k\epsilon}{N}}{k}\right)\\
&=-1+\frac{\epsilon}{2}>-1,
\end{split}\]
thus this shows that $\Vert v(\cdot,t) \Vert _{W^{1,q}(\Omega)}<C(t)$ for all $t\in(0,\infty)$ in virtue of (\ref{28}) and we conclude the proof of Lemma \ref{lem3}.
\end{proof}

Now we are ready to present the proof of our main theorem.
\begin{proof}[Proof\nopunct] \emph{of Theorem} \ref{T1}.
According to Theorem 5.2 in \cite{Am}, in order to show that $T_{\max}=\infty$, it is sufficient to prove that $\sup_{s\in(0,t)} \Vert (u,v)(\cdot,s) \Vert_{L^\infty}$ is bounded for all $t\in(0,T_{\max})$ since (\ref{4}) is a triangular system, hence we want to show that
\begin{equation}\label{29}
\Vert u(\cdot,t) \Vert_{L^\infty(\Omega)}+\Vert v(\cdot,t) \Vert_{L^\infty(\Omega)} \leq C(t),~\text{for all}~t\in(0,\infty),
\end{equation}
and we shall only need to prove the boundedness of $\Vert u(\cdot,t) \Vert_{L^\infty(\Omega)}$ for all $t\in(0,\infty)$ in virtue of Lemma \ref{lem2}.  To this end, we rewrite the $u$-equation into the following abstract form
\begin{equation}\label{30}
u(\cdot,t)=e^{t \Delta}u_0-\int_0^t  e^{(t-s)\Delta} \nabla \cdot \Big(\frac{\chi u(\cdot,s)}{v(\cdot,s)+c} \nabla v(\cdot,s) \Big) ds.
\end{equation}
Moreover, we will apply the smoothing estimate that, if $p\in(1,\infty)$ and $\theta \in(0,1)$, then for any $\epsilon>0$, there exists $C_\epsilon$ such that
\[\Vert A^\theta e^{t\Delta} \nabla \cdot w \Vert_{L^p} \leq C_\epsilon t^{-\frac{1}{2}-\theta-\epsilon} e^{-\nu t} \Vert w \Vert_{L^p},~\forall w \in L^p.\]
Then for all $1<p,~q\leq \infty$, we have from (\ref{30}) and the estimate above, for any $\epsilon>0$, there exists $C_\epsilon>0$ such that
\begin{align}\label{31}
  \begin{split}
&\Vert u(\cdot,t) \Vert_{W^{1,q}(\Omega)} \leq C \left(\Vert u_0 \Vert_{W^{1,q}(\Omega)}+ \int_0^t  \big\Vert e^{(t-s)\Delta} \nabla \cdot \Big(\frac{\chi u(\cdot,s)}{v(\cdot,s)+c} \nabla v(\cdot,s) \Big) \big\Vert_{L^p} ds \right)\\
\leq& C \left(\Vert u_0 \Vert_{W^{1,q}(\Omega)}+C_\epsilon \chi \int_0^t (t-s)^{-1-\frac{N}{2}(\frac{1}{p}-\frac{1}{q})-\epsilon} e^{-\nu (t-s)} \big\Vert  \frac{u(\cdot,s)}{v(\cdot,s)+c} \nabla v(\cdot,s) \big\Vert_{L^p} ds\right) \\
  \leq & C \left(\Vert u_0 \Vert_{W^{1,q}(\Omega)}+C_\epsilon \chi \int_0^t (t-s)^{-1-\frac{N}{2}(\frac{1}{p}-\frac{1}{q})-\epsilon} \Vert  u(\cdot,s) \nabla v(\cdot,s) \Vert_{L^p} ds \right) \\
\leq & C \left(\Vert u_0 \Vert_{W^{1,q}(\Omega)}\!+\!C_\epsilon \chi\!\! \int_0^t\!\! (t\!-\!s)^{-1-\frac{N}{2}(\frac{1}{p}-\frac{1}{q})-\epsilon} \Vert  u(\cdot,s)  \Vert_{L^{\tilde{p}}} \! \Vert\nabla v(\cdot,s)  \Vert_{L^{\frac{\tilde{p} p}{\tilde{p}-p}}} ds \right).
  \end{split}
\end{align}
where $\tilde{p}>p$ and we have used the fact
\[ \Vert  u \nabla v \Vert_{L^p} \leq  \Vert  u  \Vert_{L^{\tilde{p}}}   \Vert\nabla v  \Vert_{L^{\frac{\tilde{p} p}{\tilde{p}-p}}}, t\in(0,\infty).\]
To verify this inequality, we put $r=\frac{\tilde{p}}{p}>1$ and apply Holder's inequality to see that
\[\int_\Omega u^p \vert \nabla v \vert^p dx \leq \left(\int_\Omega u^{pr}  dx \right)^\frac{1}{r}\left(\int_\Omega \vert \nabla v \vert^{pr^*} dx \right)^\frac{1}{r^*}=\Vert u \Vert^p_{L^{pr}}   \Vert \nabla v \Vert^p_{L^{pr^*}},\]
where $r^*=\frac{\tilde{p}}{\tilde{p}-p}$ is the Sobolev conjugate of $r$.  Then the inequality above leads to the desired estimate on $\Vert u\nabla v \Vert_{L^p}$.  Again we divide our analysis into the following two cases:

if $\chi \in I^+(N)$;  we can put $\tilde{p}=\left(\Big(\chi-\frac{k-1}{2}\Big)^2/k+\frac{\epsilon}{2N}\right)^{-1}$ and choose
\[p=\Big((\chi-\frac{k-1}{2})^2/k+\frac{\epsilon}{N}\Big)^{-1},~q=\left(\Big(\chi-\frac{k-1}{2}\Big)^2/k+\frac{4\epsilon}{N}\right)^{-1},\]
then we readily see that $\tilde{p}>p$;  moreover, we can take $\epsilon$ to be sufficiently small such that $q>N$ and
\[-1-\frac{N}{2}\Big(\frac{1}{p}-\frac{1}{q}\Big)-\epsilon=-1+\frac{\epsilon}{2}>-1,\]
 then we can conclude from (\ref{31}) and the fact $W^{1,q} \hookrightarrow \L^\infty$ for $q>N$ that for each $t\in(0,\infty)$, $\Vert u(\cdot,t) \Vert_{L^\infty(\Omega)}<C(t)$ with some positive $C(t)$.

if $\chi \in I^+(N/2)\backslash I^+(N)$.  We claim that, if $\Vert u(\cdot,t) \Vert_{L^\mu(\Omega)} <C(t)$ for some $\mu \in \left(\frac{N}{2}, \Big(\frac{k}{(\chi-\frac{k-1}{2})^2}\Big)^+\right)$, then
$\Vert u(\cdot,t) \Vert_{L^{\bar{\mu}}(\Omega)} <C(t)$ for all $\bar{\mu}\leq \tilde{\mu}$, where
\begin{equation}\label{32}
\tilde{\mu}=
\left\{
\begin{array}{ll}
\Big(\frac{1}{\mu}-\left(\frac{2}{N}-\frac{(\chi-\frac{k-1}{2})^2}{k}\right)\Big)^{-1}, \text{ if } \frac{N}{2}<\mu< \Big(\frac{2}{N}-\frac{(\chi-\frac{k-1}{2})^2}{k}\Big)^{-1}, \\
+\infty, \text{ if } \mu\geq \Big(\frac{2}{N}-\frac{(\chi-\frac{k-1}{2})^2}{k}\Big)^{-1}.
\end{array}
\right.
\end{equation}
To prove this claim, we first take $\epsilon>0$ small and put $\tilde{p}=\Big(\frac{1}{\mu}+\frac{\epsilon}{2N}\Big)^{-1}$;  moreover we introduce the parameters $p$ and $q$ through
\[\frac{1}{p}=\frac{(\chi-\frac{k-1}{2})^2}{k}-\frac{1}{N}+\frac{1}{\mu}+\frac{4 \epsilon}{N},~\frac{1}{q}=\frac{1}{p}+\frac{2\epsilon}{N}.\]
To apply (\ref{13}) and (\ref{26}), we can choose $\epsilon$ to be sufficiently small and it follows from straightforward calculations that $p<\tilde{p}<\mu$, $\frac{\tilde{p} p}{\tilde{p}-p}<\frac{k}{(\chi-\frac{k-1}{2})^2-\frac{k}{N}}$ and $-1-\frac{N}{2}\Big(\frac{1}{p}-\frac{1}{q}\Big)-\epsilon>-1+\frac{\epsilon}{2}>-1$.  After applying Lemma \ref{lem2} and Lemma \ref{lem3} to (\ref{31}) with $p,\tilde{p}$ and $q$ chosen above, we readily see that
\[\Vert u(\cdot,t) \Vert_{W^{1,q}(\Omega)} \leq C(t),\text{ for all } t\in (0,\infty),\]
then we have proved the claim thanks to the embeddings that $W^{1,q} \hookrightarrow L^\frac{Nq}{N-q}$ if $q\leq N$ and $W^{1,q} \hookrightarrow L^\infty$ if $q> N$.  Now we proceed to prove Theorem \ref{T1}, showing that $\Vert u(\cdot,t) \Vert_{L^\infty}<C(t)$ by an iteration argument.  To this end, we define
\begin{equation}\label{33}
\mu_{n+1}=
\left\{
\begin{array}{ll}
\Big(\frac{1}{\mu_n}-\left(\frac{2}{N}-\frac{(\chi-\frac{k-1}{2})^2}{k}\right)\Big)^{-1},\text{if } \frac{N}{2}<\mu_n< \Big(\frac{2}{N}-\frac{(\chi-\frac{k-1}{2})^2}{k}\Big)^{-1}, \\
+\infty,\text{if } \mu_n \geq \Big(\frac{2}{N}-\frac{(\chi-\frac{k-1}{2})^2}{k}\Big)^{-1},
\end{array}
\right.
\end{equation}
since $\chi \in I^+(N/2)\backslash I^+(N)$, one can always find $\mu_1$ such that $\frac{N}{2}<\mu_1< \Big(\frac{2}{N}-\frac{(\chi-\frac{k-1}{2})^2}{k}\Big)^{-1}$.  On the other hand, we have that $\mu_{n+1}\geq \mu_n$, $n=1,2,...$ and $\mu_n \rightarrow \infty$ for sufficiently large but finite $n$.  Moreover, since $\Vert u(\cdot,t) \Vert_{L^{\mu_1}}<C(t)$, we use this fact together with (\ref{33}) iteratively to $\Vert u(\cdot,t) \Vert_{L^{\mu_n}}$ and eventually we can show that $\Vert u(\cdot,t) \Vert_{L^\infty}<C(t)$, $t\in(0,\infty)$.  This concludes the proof of Theorem \ref{T1}.
\end{proof}

\section{Conclusion and discussion}
In this paper, we investigate the global existence of a parabolic-parabolic Keller-Segel chemotaxis model (\ref{4}) under homogeneous Neumann boundary over smoothly bounded domains $\Omega\subset \mathbb{R}^N$, $N\geq2$.  The chemotactic sensitivity is chosen to be a saturated logarithmic function.  This adaptation is made according to the well-accepted Weber-Fechner's law of stimulus perception.  Global existence of (\ref{4}) are established provided with conditions on $k$ and $\chi$.-see Theorem \ref{T1}.  Formally speaking, (\ref{4}) admits a unique nonnegative global solution if $\chi$ is not too large for all $k\in(k_1,k_2)$, where $k_1$ and $k_2$ are introduced in Definition \ref{def22} and $0<k_1<1<k_2$.  This generalizes the results in the literature for $k=1$, which can not be dropped due to technical issues.

In our proof of Lemma \ref{lem2} and that of Theorem \ref{T1}, we require that $\chi\in(0,\chi_2)$ for $k\in(k_1,k_2)$ and $\chi\in(\chi_1,\chi_2)$ for $k\in[k_2,\infty)$.  It is necessary to point out that, these conditions are imposed only due to technical reasons and we believe that both requirements can be relaxed, at the least the unusual lower bound on $\chi$ when $k\in[k_2,\infty)$.  Actually, for $\chi=0$, (\ref{4}) clearly admits global solutions for all $k\in(0,\infty)$, therefore, it is quite reasonable to believe that it still holds for $\chi$ being small.  New techniques are needed to address this issue.

There are apparently many important and interesting unsolved questions regarding the global existence of (\ref{4}).  Besides those discussed above, we list some others here.  Theorem \ref{T1} merely states that solutions exist globally and the finite time blow-up is impossible.  However, since the solution bound depends on the time, it is unclear whether infinite time blow-up can be excluded or not, even when provided with the assumptions on $\chi$ and $k$.  Though a mathematically strict proof is not presented here, we surmise that solution bound is uniform in time.  Actually, the cross-diffusion term $\frac{u \nabla v}{v+c}$ advects down the directed chemotactic movement of cells as the chemical density increases, then the diffusion term dominates in the $u$-equation and it smoothes $u$ hence $\nabla v $ in the $v$-equation.  The dynamical behavior of the global solutions of (\ref{4}) can also be an interesting problem to work on.  The existence or non-existence of nontrivial steady states of (\ref{4}) also deserves exploring in the future.

\section*{Acknowledgments}
The author would like to thank Prof. Xuefeng Wang at Tulane University for his valuable comments and suggestions which greatly improved the exposition of this paper.  The author is also grateful for various insightful and constructive suggestions of the anonymous referee, which greatly improved the exposition of this manuscript.  He also wants to convey thanks to Brendon Enlow for carefully reading and polishing the writing of this paper.  This research is supported by the Fundamental Research Funds for the Central Universities, China and the Project-sponsored by SRF for ROCS, SEM.


\medskip
\medskip

\end{document}